\documentclass[10pt]{amsart}

\usepackage{verbatim}

\usepackage{amsthm}
\usepackage{amsfonts}
\usepackage{amsmath}
\usepackage{amssymb}
\usepackage{verbatim, color}

\usepackage{pictex}
\usepackage{url}

\begin{document}
 \newcounter{thlistctr}
 \newenvironment{thlist}{\
 \begin{list}%
 {\alph{thlistctr}}%
 {\setlength{\labelwidth}{2ex}%
 \setlength{\labelsep}{1ex}%
 \setlength{\leftmargin}{6ex}%
 \renewcommand{\makelabel}[1]{\makebox[\labelwidth][r]{\rm (##1)}}%
 \usecounter{thlistctr}}}%
 {\end{list}}

\thispagestyle{empty}

\newtheorem{Lemma}{\bf LEMMA}[section]
\newtheorem{Theorem}[Lemma]{\bf THEOREM}
\newtheorem{Claim}[Lemma]{\bf CLAIM}
\newtheorem{Corollary}[Lemma]{\bf COROLLARY}
\newtheorem{Proposition}[Lemma]{\bf PROPOSITION}
\newtheorem{Example}[Lemma]{\bf EXAMPLE}
\newtheorem{Fact}[Lemma]{\bf FACT}
\newtheorem{definition}[Lemma]{\bf DEFINITION}
\newtheorem{Notation}[Lemma]{\bf NOTATION}
\newtheorem{remark}[Lemma]{\bf REMARK}

\newcommand{\restrict}{\mbox{$\mid\hspace{-1.1mm}\grave{}$}}
\newcommand{\covers}{\mbox{$>\hspace{-2.0mm}-{}$}}
\newcommand{\covered}{\mbox{$-\hspace{-2.0mm}<{}$}}
\newcommand{\notcover}{\mbox{$>\hspace{-2.0mm}\not -{}$}}

\newcommand{\boldalpha}{\mbox{\boldmath $\alpha$}}
\newcommand{\boldbeta}{\mbox{\boldmath $\beta$}}
\newcommand{\boldgamma}{\mbox{\boldmath $\gamma$}}
\newcommand{\boldxi}{\mbox{\boldmath $\xi$}}
\newcommand{\boldlambda}{\mbox{\boldmath $\lambda$}}
\newcommand{\boldmu}{\mbox{\boldmath $\mu$}}

\newcommand{\barzero}{\bar{0}}

\newcommand{\sfq}{{\sf q}}
\newcommand{\sfe}{{\sf e}}
\newcommand{\sfk}{{\sf k}}
\newcommand{\sfr}{{\sf r}}
\newcommand{\sfc}{{\sf c}}
\newcommand{\restr}{\negmedspace\upharpoonright\negmedspace}

\title[De Morgan Semi-Heyting Algebras]{De Morgan Semi-Heyting and Heyting Algebras}             

\author{Hanamantagouda P. Sankappanavar}

\keywords{De Morgan semi-Heyting algebras of level $n$, subvarieties 
 De Morgan Heyting algebra, dually Stone Heyting algebra, 
discriminator variety, simple algebra,  
subdirectly irreducible algebra, equational base.} 

\subjclass[2010]{$Primary:03G25, 06D20, 08B15, 06D15, 03C05, 03B50;$  $Secondary:08B26, 06D30, 06E75$}

\begin{abstract}

The variety $\mathbf{DMSH}$ of semi-Heyting algebras with a De Morgan negation
was introduced in ~\cite{Sa11} and an increasing sequence $\mathbf{DMSH_n}$ of level $n$, for $n \in \omega$, of its subvarieties was investigated in the series ~\cite{Sa11}, 
~\cite{Sa14}, ~\cite{Sa14a}, ~\cite{Sa15}, \cite{Sa17}, and \cite{Sa18}, of which the present paper is a sequel.
In this paper, we prove two main results:  
Firstly, we prove that $\mathbf{DMSH_1}$-algebras of level $1$ satisfy Stone identity, generalizing an earlier result that regular $\mathbf{DMSH_1}$-algebras of level $1$ satisfy Stone identity.
Secondly, we prove that the variety of $\mathbf{DmsStSH}$ of dually ms, Stone semi-Heyting algebras is at level $2$.   As an application, it is derived that the variety of De Morgan semi-Heyting algebras is also at level $2$.  It is also shown that these results are sharp.  
\end{abstract}

\maketitle

\thispagestyle{empty}

\section{{\bf Introduction}} \label{SA}

The variety $\mathbf{DQDSH}$ of semi-Heyting algebras 
with a dually quasi-De Morgan negation
was introduced and investigated in ~\cite{Sa11}. 
Several important subvarieties of $\mathbf{DQDSH}$
were also introduced in the same paper, some of which are the following: Subvarieties of $\mathbf {DQD}$ of level $n$, for $n \in \mathbf{\omega}$,  
the variety $\mathbf {DMSH}$ of   
De Morgan (symmetric) semi-Heyting algebra, and DmsStSH of dually ms Stone semi-Heyting algebras.  
The work of  \cite{Sa11} was continued in \cite{Sa14}, \cite{Sa14a}, \cite{Sa15},  \cite{Sa17}
and \cite{Sa18}.   
We also note that the variety $\mathbf{DMSH}$ is an equivalent algebraic semantics for the propositional logic, called ``De Morgan semi-Heyting logic'' which is recently introduced in \cite{CoSa18}.  Since the lattice of subvarieties of $\mathbf{DMSH}$ is dually isomorphic to the lattice of extensions of the De Morgan semi-Heyting logic, the results of this paper and the earlier papers in this series have logical counterparts in the corresponding (logical) extensions of the De Morgan semi-Heyting logic; and, in particular, to De Morgan Heyting logic.

In this paper, we present two main results: Firstly, we prove that the variety of De Morgan semi-Heyting algebras of leve $1$ ($\mathbf{DMSH_1}$ for short) satisfies Stone identity, generalizing an earlier result that regular $\mathbf{DMSH_1}$-algebras of level $1$ satisfy Stone identity, proved in \cite{Sa15}.
Secondly, we prove that the variety of $\mathbf{DmsStSH}$ of dually ms, Stone semi-Heyting algebras is at level 2.   As an application, it is derived that the variety of De Morgan Stone semi-Heyting algebras is at level 2.  Finally, it is shown that these results are sharp.

\vspace{1cm}
\section{\bf {Preliminaries}} \label{SB}

In this section we recall some definitions and results needed in this paper.   
For other relevant information, we refer the reader to the textbooks \cite{BaDw74}, \cite{BuSa81} and \cite{Ra74}.

An algebra ${\mathbf L}= \langle L, \vee ,\wedge ,\to,0,1 \rangle$
is a {\it semi-Heyting algebra} (\cite{Sa07}) if \\
 $\langle L,\vee ,\wedge ,0,1 \rangle$ is a bounded lattice and ${\mathbf L}$ satisfies:
\begin{enumerate}
\item[{\rm(SH1)}] $x \wedge (x \to y) \approx x \wedge y$,
\item[{\rm(SH2)}] $x \wedge(y  \to z) \approx x \wedge [(x \wedge y) \to (x \wedge z)]$,
\item[{\rm(SH3)}] $x \to x \approx 1$.
\end{enumerate}
Semi-Heyting algebras are distributive and pseudocomplemented,
with $a^* := a \to 0$ as the pseudocomplement of an element $a$.  These and other properties (see \cite{Sa07}) of semi-Heyting algebras are frequently used without explicit mention throughout this paper.
 
Let ${\mathbf L}$ be a semi-Heyting algebra.     
${\mathbf L}$ is a {\it Heyting algebra} if ${\mathbf L}$ satisfies:
\begin{enumerate}
\item[{\rm(H)}] $(x \wedge y) \to y \approx 1$.
\end{enumerate}
$\mathbf{L}$ is a {\it Stone semi-Heyting algebra} 
if ${\mathbf L}$ satisfies:
\begin{enumerate}
\item[{\rm(St)}] $x^* \lor x^{**} \approx 1$.
\end{enumerate}
The variety of Stone semi-Heyting algebras is denoted by $\mathbf{StSH}$ or just by $\mathbf{St}$.

The following definition, taken from \cite{Sa11}, is central to this paper.
\begin{definition}
An algebra ${\mathbf L}= \langle L, \vee ,\wedge ,\to, ', 0,1
\rangle $ is a {\it semi-Heyting algebra with a dual quasi-De Morgan
operation} or {\it dually quasi-De Morgan semi-Heyting algebra}
\rm{(}$\mathbf {DQD}$-algebra, for short\rm{)}  if\\
 $\langle L, \vee ,\wedge ,\to, 0,1 \rangle $ is a semi-Heyting algebra, and
${\mathbf L}$ satisfies:
 \begin{itemize}
    \item[(a)]  $0' \approx 1$ and $1' \approx 0$,
     \item[(b)] 
      $(x \land y)' \approx x' \lor y'$,
    \item[(c)]
    $(x \lor y)''  \approx x'' \lor y''$,     
    \item[(d)]  
   $x'' \leq x$.
\end{itemize}
\noindent 
Let $\mathbf{L}$ be a $\mathbf {DQD}$-algebra.  
 $\mathbf {L}$ is a {\it De Morgan semi-Heyting algebra}  
  {\rm(}$\mathbf{DM}$-algebra{\rm)} if ${\bf
L}$ satisfies:
\begin{itemize}
\item[(DM)]  $x'' \approx x$.
\end{itemize}
${\bf L}$ is a {\it dually ms semi-Heyting algebra} {\rm(}$\mathbf{Dms}$-algebra{\rm)} if ${\bf
L}$ satisfies: 
\begin{itemize}
\item [(JDM)] $(x \lor y)' \approx x' \land y'$  ($\lor$-De Morgan law).
\end{itemize} 
 $\mathbf{L}$ is a {\it blended dually quasi-De Morgan semi-Heyting algebra}  {\rm(}$\mathbf{BDQD}$-algebra{\rm)} if  $\mathbf{L}$ satisfies the
following identity: 
 \begin{itemize}
\item[(BDM)]  $(x \lor x^*)' \approx x' \land x{^*}'$  (Blended $\lor$-De Morgan law).
 \end{itemize}
 $\mathbf{L}$ is regular if  $\mathbf{L}$ satisfies the
following identity: 
 \begin{itemize}
 \item[(R)]  $ x \land x^+ \leq y \lor y^*$, where $x^+ := x'{^*}'$. 
 \end{itemize}
\end{definition}  
\emph {The reader should be cautioned that this notion of regularity is totally different from the one used in \cite{Sa11}.}

The varieties of $\mathbf {DQD}$-algebras,  
$\mathbf{Dms}$-algebras,  
and $\mathbf{DM}$-algebras 
 are denoted, respectively, by $\mathbf {DQD}$,  $\mathbf{Dms}$,
and $\mathbf{DM}$.
 The variety of regular DQD-algebras will be denoted by $\mathbf{RDQD}$.
$\mathbf {DQDSt}$ denotes the subvariety of $\mathbf{DQD}$ with Stone semi-Heyting reducts.        
If the underlying semi-Heyting algebra of a $\mathbf {DQD}$-algebra is a Heyting
algebra, then we add    
``$\mathbf{H}$'' at the end of the names of the varieties that will be considered in the sequel; for example, $\mathbf {DQDH}$ denotes the variety of dually quasi-De Morgan Heyting algebras.

\begin{Lemma} \label{2.2}
Let ${\mathbf L} \in \mathbf{DQD}$ and let $x,y, z \in L$.  Then
\begin{enumerate}
 \item[{\rm(i)}]  $1'^{*}=1$, and $1 \to x =x$,
\item[{\rm(ii)}] $x \leq y$ implies $x' \geq y'$,  
\item[{\rm(iii)}] $(x \land y)'^{*}=x'^{*} \land y'^{*}$,
\item[{\rm(iv)}] $ x''' = x'$,
\item[{\rm(v)}] $x \lor x^+ = 1$. 
\end{enumerate}
\end{Lemma}

The following definition is from \cite{Sa11}; it helps us to classify subvarieties of $\mathbf{DQD}$ by means of ``levels''.  It also plays a crucial role in describing an increasing sequence of discriminator subvarieties of $\mathbf{DQD}$ (see \cite{Sa11} for more details).

\begin{definition}\label{5.5}

Let $\mathbf{L} \in \mathbf{DQD}$ and $x \in {\bf L}$.  For $n \in \omega$, we
define $t_n(x)$ recursively as follows:\par

\begin{center}
$x{^{0(}{'{^{*)}}}} := x$;  \\
$x^{(n+1)(}{'{^{*)}}} := (x{^{n(}}{'{^{*)}}})'{^*}$, for $n \geq 0$;\\
$t_0(x) := x$, \\
     $t_{n+1}(x) := t_n(x) \land x^{(n+1){\rm(}}{'^{*{\rm{)}}}}$, for $n \geq 0$.
\end{center}

Let $n \in \omega$. 
The subvariety $\mathbf {DQD_n}$ {\it of level $n$}
of $\mathbf {DQD}$ is defined by the identity:
\begin{equation}
t_n(x)\approx t_{(n+1)}(x); \tag{L$_n$}
\end{equation}
For a subvariety $\mathbf {V}$ of $\mathbf {DQD}$, we let $\mathbf {V_n}:= \mathbf {V} \cap \mathbf {DQD_n}$.  
We say that $\mathbf {V_n}$ is ``at level $n$'' or ``of level $n$''.
\end{definition}

It should be noted that, for $n \in \omega$, 
the variety $\mathbf {BDQD_n}$, is a discriminator variety (\cite{Sa11}).

\vspace{1cm} 
 
\section{The variety $\mathbf{DM_1}$ of De Morgan semi-Heyting algebras of level $1$}

It was proved in \cite{Sa15} that regular De Morgan semi-Heyting algebras of level 1 satisfy the Stone identity.  The purpose of this section is to prove a more general theorem which says that De Morgan semi-Heyting algebras of level 1 satisfy the Stone identity. 

The following Lemma, whose proof is immediate since $x{^{2(}}{'^{*)}} \leq x$ in a 
$\mathbf{DMSH}$-algebra, gives an alternate definition of ``level $n$'', for $n \in \mathbb{N}$.

\begin{Lemma}
Let $\mathbf{L} \in \mathbf{DM}$ and let $n \in \omega$.  Then $\mathbf{L}$ is at level $n+1$ iff $\mathbf{L}$ satisfies the identity:
\begin{itemize}
\item[(L'$_n$)] \quad $(x \land x'^*){^{n}{^(}}{'^{*)}} \approx (x \land x'^*){^{(n+1)}{^(}}{'^{*)}}$.
\end{itemize}
\end{Lemma}

Recall that $\mathbf{DM_1}$ denotes $\mathbf{DMSH} \cap \mathbf{DQD_1}$.

{\bf In the rest of this section $\mathbf{L} \in \mathbf{DM_1}$ and $x \in L$.}

\begin{Lemma} \label{110}
Let $\mathbf{L} \in \mathbf{DM_1}$ and $x \in L$.  Then,
$x{^*}' \land x' \land x^* =0.$
\end{Lemma}

\begin{proof}  
$$
\begin{array}{lcll}
x{^*}' \land x' \land x^* 
                                         &=& x{^*}' \land (x' \land x''^*)'^*  &\text{by Lev 1 and (DM)} \\ 
                                         & =& x{^*}' \land (x'' \lor x''{^*}')^* \\
                                         & =& x{^*}' \land (x \lor x{^*}')^* &\text{by (DM)} \\  
                                        & \leq& x{^*}' \land  x{^*}'^* \\
                                         & = &0. &
\end{array}
$$
\end{proof}

We are ready to present our main theorem of this section.  

\begin{Theorem} \label{DMSt}
Let $\mathbf{L} \in \mathbf{DM_1}$.  Then
$\mathbf{L}$ satisfies Stone identity.
\end{Theorem}

\begin{proof} 
We have $x{^*}' \land x{^{**}}' \land x^* 
= x{^*}' \land x{^{**}}' \land (x^{**} \to 0)
= x{^*}' \land x{^{**}}' \land [(x^{**} \land x{^*}' \land  x{^{**}}')  \to 0]$,
which, by Lemma \ref{110}, implies  $x{^*}' \land x{^{**}}' \land x^* = x{^*}' \land x{^{**}}' $,         
from which, in view of Lemma \ref{110}, we get $x{^*}' \land x{^{**}}' =0.$  Hence, $x{^*}'' \lor x{^{**}}'' =1.$  which, by (DM),  implies   $x^* \lor x^{**} =1$.
\end{proof}

\begin{Corollary}
$\mathbf{DM_1}= \mathbf{DMSt_1}$.
In particular, 
$\mathbf{DMH_1}= \mathbf{DMStH_1}$.
\end{Corollary}

Recall that $\mathbf{RDM_1}$ denotes the variety of regular De Morgan semi-Heyting algebras of level $1$.  The following result, proved in \cite[Theorem 3.8]{Sa15}, is now a special case of the preceding corollary.
\begin{Corollary}{\rm \cite[Theorem 3.8]{Sa15}}
$\mathbf{RDM_1}= \mathbf{RDMSt_1}$.
\end{Corollary}

It is natural to wonder if Theorem \ref{DMSt} can be further generalized.  There are two possible directions to try to generalize, which led us to ask if the theorem holds in $\mathbf{DM_2}$ or in $\mathbf{Dms_1}$.  However, the theorem fails in both cases, as shown by the following two examples:

\begin{Example} \label{Ex}
 Theorem \ref{DMSt} fails in the variety $\mathbf{DM_2}$ as witnessed by 
 the following algebra:

 Let $\mathbf{A} = \langle A, \lor, \land, \to, ', 0,1 \rangle$be an algebra, with $A= \{0, a, b, c, d, e, 1 \} $, whose lattice reduct and the operatons $\to$ and $'$ are defined in Figure 1.  It is routine to verify that that $\mathbf{A} \in \mathbf{DM_2}$, 
 but fails to satisfy the Stone identity \rm(at $b$\rm).
\end{Example} 
\vspace{.5cm}
\setlength{\unitlength}{.7cm} 
\begin{picture} (8,1)  

\put(5,1){\circle*{.15}}  
\put(4.5,1){$1$}

\put(4,0){\circle*{.15}}  
\put(3.5,0){$e$}  

\put(5.2,-1.1){$a$}  

\put(6,0){\circle*{.15}} 

\put(6.2,0){$c$}

\put(5,-1){\circle*{.15}} 

\put(4,-2){\circle*{.15}}  
\put(3.4, -2.1){$d$}  

\put(6,-2){\circle*{.15}}  
\put(6.2, -2.1){$b$}  

\put(5,-3){\circle*{.15}} 
\put(5.2,-3.3){$0$}

\put(4,-2){\line(1,1){1}}
\put(5,-3){\line(1,1){1}}

\put(4,0){\line(1,-1){1}}
\put(5,-1){\line(1,-1){1}}
\put(4,-2){\line(1,-1){1}}
                               \put(5,-1){\line(1,1){1}}

\put(4,0){\line(1,1){1}}
\put(5,1){\line(1,-1){1}}

\put(5,-10.0){Figure 1}      
\end{picture}

\vspace{2.5cm}

\begin{tabular}{r|rrrrrrr}
$\to$: & $0$ & $1$ & $d$ & $e$ & $b$ & $c$ & $a$\\
\hline
    $0$ & $1$ & $1$ & $1$ & $1$ & $1$ & $1$ & $1$ \\
    $1$ & $0$ & $1$ & $d$ & $e$ & $b$ & $c$ & $a$ \\
    $d$ & $b$ & $1$ & $1$ & $1$ & $b$ & $1$ & $1$ \\
    $e$ & $0$ & $1$ & $d$ & $1$ & $b$ & $c$ & $c$ \\
    $b$ & $d$ & $1$ & $d$ & $1$ & $1$ & $1$ & $1$ \\
    $c$ & $0$ & $1$ & $d$ & $e$ & $b$ & $1$ & $e$ \\
    $a$ & $0$ & $1$ & $d$ & $1$ & $b$ & $1$ & $1$
\end{tabular} \hspace{.5cm}
 \begin{tabular}{r|rrrrrrr}
$'$: & $0$ & $1$ & $d$ & $e$ & $b$ & $c$ & $a$\\
\hline
   & $1$ & $0$ & $e$ & $d$ & $c$ & $b$ & $a$
\end{tabular} \hspace{.5cm}\\

\vspace{1.5cm}
Thus, $\mathbf{DMSt_2} \subset \mathbf{DM_2}$.    
Also, since the algebra in Figure 1 is actually in $\mathbf{RDM_2}$, we also have $\mathbf{RDMSt_2}  \subset  \mathbf{RDM_2}$.

\begin{Example}  The following algebra is in $\mathbf{RDms_1}$,  but fails to satisfy the Stone identity.\\
\vspace{3cm}
\setlength{\unitlength}{.7cm} 
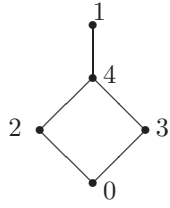
\begin{picture} (8,1)  
\put(5,0){\circle*{.15}}  
\put(5,.1){$1$}


\put(5.2,-1.1){$4$}  



\put(5,-1){\circle*{.15}} 

\put(4,-2){\circle*{.15}}  
\put(3.4, -2.1){$2$}  

\put(6,-2){\circle*{.15}}  
\put(6.2, -2.1){$3$}  

\put(5,-3){\circle*{.15}} 
\put(5.2,-3.3){$0$}

\put(4,-2){\line(1,1){1}}
\put(5,-3){\line(1,1){1}}

        \put(5,-1){\line(0,1){1}}

\put(5,-1){\line(1,-1){1}}
\put(4,-2){\line(1,-1){1}}


\put(5,-10.5){{\rm Figure} $2$}      
\end{picture}

\begin{tabular}{r|rrrrr}
$'$: & $0$ & $1$ & $2$ & $3$ & $4$\\
\hline
   & $1$ & $0$ & $1$ & $1$ & $1$
\end{tabular} \hspace{.5cm}\\
\ \\ \ \\
\begin{tabular}{r|rrrrr}
$\to$: & $0$ & $1$ & $2$ & $3$ & $4$\\
\hline
    $0$ & $1$ & $0$ & $3$ & $2$ & $0$ \\
    $1$ & $0$ & $1$ & $2$ & $3$ & $4$ \\
    $2$ & $3$ & $2$ & $1$ & $0$ & $2$ \\
   $3$ & $2$ & $3$ & $0$ & $1$ & $3$ \\
    $4$ & $0$ & $1$ & $2$ & $3$ & $1$
\end{tabular} \hspace{.5cm}
\end{Example}
\vspace{1cm}
Thus, $\mathbf{DmsSt_1} \subset \mathbf{Dms_1}$.  In fact, $\mathbf{RDmsSt_1} \subset \mathbf{RDms_1}$.

\vspace{1cm}
\section{The level of the variety $\mathbf{DmsSt}$ of dually ms, Stone semi-Heyting algebras}

Recall that the variety of dually ms Stone semi-Heyting algebras, $\mathbf{DmsSt}$, is defined, relative to $\mathbf{DQD}$, by the following identities:
\begin{itemize}
\item[(St)]  \quad $x^* \lor x^{**} \approx 1$,
\item[{\rm (JDM)}] \quad $(x \lor y)' \approx x' \land y'$ \qquad ($\lor$-De Morgan law).  
\end{itemize}

Since $\mathbf{DQD}$-algebras satisfy $x'' \leq x$, it is clear that the $\{\lor, \land, ' \}$-reduct of a $\mathbf{DQD}$-algebra satisfying (JDM) is indeed a dual ms algebra.

In this section our goal is to show that the variety $\mathbf{DmsSt}$ lies at level 2, but not at level 1.  
The following theorem, proved in \cite[Theorem 2.5]{Sa14}, 
gives an alternative definition of $\mathbf {DQDSt_n}$. 

\begin{Theorem}{\rm \cite[Theorem 2.5]{Sa14}} \label{level_second form}
For $n \in \omega$, $\mathbf {DQDSt_{n+1}}$ is defined by the identity:
 $(x \land x'{^*})^{n('{^*})} \approx (x \land x'{^*})^{(n+1)('{^*})}$, relative to $\mathbf {DQDSt}$.

In particular, the variety  $\mathbf {DQDSt_1}$ and $\mathbf {DQDSt_2}$ are defined, respectively, by the identities:
$x \land x'{^*} \approx (x \land x'^*)'{^*}$ and $(x \land x'{^*})'^* \approx (x \land x'^*)'{^*}'^*$, relative to $\mathbf {DQDSt}$. 
\end{Theorem}

{\bf Throughout this section,  $\mathbf{L}$ denotes an algebra in $\mathbf{DmsSt}$, and $x \in \mathbf{L}$. }   

Let $x^+ := x'{^*}'$.
\begin{Lemma} \label{Level_Lemma3} 
$x'  \lor  x{^*}'{^{**}} =1$.
\end{Lemma}

\begin{proof}
Since $x'  \lor x{^*}'{^*}= (x'  \lor x{^*}'{^*}) \land (x \land x^*)' = (x'  \lor x{^*}'{^*}) \land (x' \lor x{^*}' )= x'  \lor (x{^*}'{^*} \land x{^*}' )  =x' $, we get
\begin{equation} \label{eqnA}
 x{^*}'{^*} \leq x'.
\end{equation}
Hence, $x'  \lor  x{^*}'{^{**}}= x'  \lor x{^*}'{^{*}} \lor  x{^*}'{^{**}}=1$, in view of \eqref{eqnA} and (St). 
\end{proof}

\begin{Lemma} \label{Level_Lemma5} 
$x{^*}' \leq x{^{**}}'^{*}$.
\end{Lemma}

\begin{proof}  
$x{^*}' \land x{^{**}}'^{*}
= x{^*}' \land (x{^{**}}' \to 0)
= x{^*}' \land [(x{^*}'  \land x{^{**}}') \to (x{^*}' \land 0)]
= x{^*}' \land [(x{^*}'  \land x{^{**}}') \to 0]
= x{^*}' \land [(x{^*}  \lor x^{**})' \to 0] 
= x{^*}' \land (0 \to 0)   
= x{^*}' $, in view of  the identities (JDM) and (St). 
\end{proof}

\begin{Lemma} \label{Level_Lemma7} 
$x{^{**}}' = x{^{*}}'^{*}$.  Hence, $x{^*}{^+} \leq x^{**}$. 
\end{Lemma}

\begin{proof} From the equation \eqref{eqnA} of           
Lemma \ref{Level_Lemma3} and Lemma \ref{Level_Lemma5}, we have $x{^*}' = x{^{**}}'^{*}$, from which, replacing $x$ by $x^*$, we get
$x{^{**}}' =  x{^{***}}'^*$, 
leading to $x{^{**}}' = x{^{*}}'^{*}$. Hence, $x{^*}{^+}=  x{^{**}}''  \leq x^{**}$.
\end{proof}

\begin{Lemma} \label{Level_Lemma9} 
$x{^*}'' = x^{*}$.
\end{Lemma}

\begin{proof} 
From $x^* \land x{^{**}}'' =0$, we have $(x{^*}'' \lor x^*) \land (x{^{*}}'' \lor x{^{**}}'') = x{^*}''$, which implies $(x{^*}'' \lor x^*) \land (x^{*} \lor x{^{**}})'' = x{^*}''$, whence, by (St), we get $x{^*}'' \lor x^* = x{^*}''$.  Thus we can conclude that $x{^*}'' = x^{*}$.
\end{proof}

\begin{Lemma} \label{Level_Lemma10} 
$x{^*}''  \leq x'{^*}' $.
\end{Lemma}

\begin{proof}
From $x  \land  x{^*}'' =0$ we have $(x \lor  x'{^*}' ) \land (x{^*}''  \lor x'{^*}'{\color{red})}  = x'{^*}'$, implying that $x{^*}''  \leq x'{^*}' $, in view of Lemma \ref{2.2} (v). 
\end{proof}

\begin{Lemma} \label{Level_Lemma12} 
$x{^*}^+  = x^{**} $.
\end{Lemma}

\begin{proof}
Applying Lemma \ref{Level_Lemma9} and Lemma \ref{Level_Lemma10}, we get $x^*  \leq x^+ $.
Then, from  
Lemma \ref{Level_Lemma7}, we have $x^{**} \leq x{^*}{^+} \leq x^{**}$.  
\end{proof}

We are now ready to prove our main theorem of this section. 

\begin{Theorem}\label{Level_Theorem}
The variety $\mathbf{DmsSt}$ is at level 2. 
\end{Theorem}

\begin{proof} 
Let $\mathbf L \in \mathbf{DmsSt}$ and
let $a \in L$.  
Then,
$(a \land a'^*)'{^*}'{^*}      
=(a'^* \land a'{^*}'{^*})'^*
= (a'{^*}' \lor a'{^*}{^+})^*   
= (a^+ \lor a'^{**})^*$  by              
Lemma \ref{Level_Lemma12}.   
Hence, $(a \land a'^*)'{^*}'{^*} =a^{+*} \land a'^{***}
=a^{+*} \land a'^{*}
=(a^{+} \lor a')^{*}
=(a'^* \land a)'^*$, in view of (JDM).
\end{proof}

Recall that $\mathbf{DmsStH}$ denotes the variety of dually ms, Stone Heyting algebras.

\begin{Corollary}
$\mathbf{DmsSt} = \mathbf{DmsSt_2}$.   In particular, \\
$\mathbf{DmsStH}= \mathbf{DmsStH_2}$.
\end{Corollary}

The following corollary is now immediate.

\begin{Corollary}
The variety $\mathbf{DMSt}$ of De Morgan Stone semi-Heyting algebras is at level 2.
\end{Corollary}

\begin{proof}
Observe that $\mathbf{DMSt} \subset \mathbf{DmsSt}$ and apply Theorem \ref{Level_Theorem}.
\end{proof}

We note that Theorem \ref{Level_Theorem} is sharp in the sense that it fails to satisfy the level 1 identity, as shown by the following example.

\begin{Example} It is easy to see that the following algebra is in $\mathbf{DmsSt}$; 
but fails to satisfy the level 1 identity: $x \land x'^*= (x \land x'^*)'^*$ at $a$.\\
\end{Example}

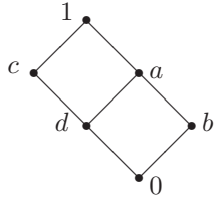
\setlength{\unitlength}{.7cm} 
\begin{picture} (8,1)  

\put(4,0){\circle*{.15}}  
\put(3.5,0){$1$}  

\put(5.2,-1.1){$a$}  

\put(3,-1){\circle*{.15}} 
\put(2.5,-1){$c$}
\put(5,-1){\circle*{.15}} 

\put(4,-2){\circle*{.15}}  
\put(3.4, -2.1){$d$}  

\put(6,-2){\circle*{.15}}  
\put(6.2, -2.1){$b$}  

\put(5,-3){\circle*{.15}} 
\put(5.2,-3.3){$0$}

\put(4,-2){\line(1,1){1}}
\put(5,-3){\line(1,1){1}}

\put(4,0){\line(1,-1){1}}
\put(5,-1){\line(1,-1){1}}
\put(4,-2){\line(1,-1){1}}

\put(3,-1){\line(1,1){1}}
\put(3,-1){\line(1,-1){1}}

\put(6.5,-8.5){Figure 3}      
\end{picture}

\vspace{2.8cm}

\begin{tabular}{r|rrrrrr}
$\to$: & $0$ & $1$ & $a$ & $b$ & $c$ & $d$\\
\hline
    $0$ & $1$ & $0$ & $0$ & $c$ & $b$ & $b$ \\
    $1$ & $0$ & $1$ &$a$ & $b$ & $c$ & $d$ \\
    $a$ & $0$ & $1$ & $1$ & $b$ & $c$ & $c$ \\
   $b$ & $c$ & $b$ & $b$ & $1$ & $0$ & $0$ \\
   $c$ & $b$ & $c$ & $d$ & $0$ & $1$ & $a$ \\
   $d$ & $b$ & $c$ & $c$ & $0$ & $1$ & $1$
\end{tabular} \hspace{.7cm}
\begin{tabular}{r|rrrrrr}
': & $0$ & $1$ & $a$ & $b$ & $c$ & $d$\\
\hline
   & $1$ &$0$ & $b$ & $b$ & $c$ & $1$
\end{tabular} \hspace{.5cm} \\

\vspace{.8cm}

Let $\mathbf{DmsL}$ denote the subvariety of $\mathbf{Dms}$ defined by the following Lee's identity (\cite{Le70}):\\

(L) \quad  $(x \land y)^* \lor  (x^*  \land y)^* \lor (x \land y^*)^* \approx 1$.

\begin{Example} The following $15$-element algebra whose universe is:\\
 $ \{0,1,2,3,4,5,6,7,8,9,10,11,12,,13,14 \}$, whose lattice reduct and the operations $\to$, $'$ are given below, is in $\mathbf{DmsL} \setminus \mathbf{DmsSt}$ 
 and is at level $\geq 3$, as the level 2 identity, $(x \land x'^*)'^*  \approx (x \land x'^*)'{^*}'^*$, fails at $2$.  Thus, the variety 
$\mathbf{DmsL}$ is at a level $\geq 3$.  \rm(We suspect, however, that its level is $3$.\rm)\\
\end{Example}

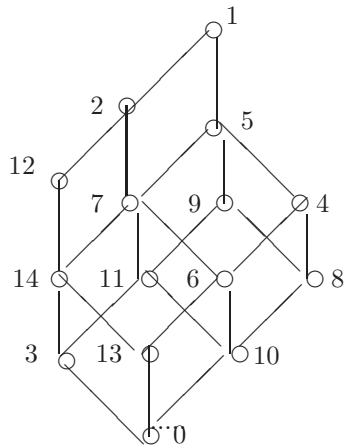
\begin{minipage}{0.5 \textwidth}
	\setlength{\unitlength}{1mm}        
	\begin{picture}(50,70)(-40,0)      
	\put(0.2,-1){\circle{2}}    
	
	\put(-11,9){\circle{2}}  
	\put(.1,10){\circle{2}}  
	\put(12,10){\circle{2}}   
	
	\put(22,20){\circle{2}}  
	\put(10,20){\circle{2}}  
		
	\put(0.0,20){\circle{2}} 
	\put(10,30){\circle{2}} 
	
	\put(20,30){\circle{2}} 
	
	\put(-12,20){\circle{2}} 
	
	\put(-2.6,30){\circle{2}}   
	\put(8.5,40){\circle{2}}  

	\put(8.5,53){\circle{2}}  
	\put(-3,43){\circle{2}}   
	\put(-12,33){\circle{2}}  
	
	\put(4,-0.8){\makebox(0,0){$0$}}
	\put(-15.7,10){\makebox(0,0){$3$}}
	\put(-5.4,10){\makebox(0,0){$13$}}  
	\put(15.5,9.8){\makebox(0,0){$10$}}
	\put(6,30){\makebox(0,0){$9$}}      
	\put(-5.0,20){\makebox(0,0){$11$}}    
	\put(5.8,20){\makebox(0,0){$6$}} 
	\put(-16.5,20){\makebox(0,0){$14$}} 
	\put(-7,30){\makebox(0,0){$7$}} 
	\put(-7,43){\makebox(0,0){$2$}} 
	 \put(-17,35){\makebox(0,0){$12$}}  
	  \put(11,55){\makebox(0,0){$1$}}   
	  
	  \put(25,20){\makebox(0,0){$8$}}
	   \put(23,30){\makebox(0,0){$4$}}
	    \put(13,41){\makebox(0,0){$5$}}
	\put(1,1){\line(1,1){8.8}}  
	\put(-11.8,9.6){\line(1,1){10.3}}    
	\put(-0.8,-2){\line(-1,1){10.5}}   
	\put(9.5,11){\line(-1,1){10}}  
	
	\put(.4,21){\line(1,1){8.5}}. 
	\put(11,11){\line(1,1){8.2}}  
	\put(19,21){\line(-1,1){8.2}}. 
	
         \put(-4.0, 10.5){\line(-1,1){10}}  
	\put(7.1,20.2){\line(-1,1){10.0}}   
	\put(17.2,31){\line(-1,1){9.8}}    
	
	\put(9.2,21){\line(1,1){9.8}}  
	\put(-3,10){\line(1,1){10}}. 
	\put(-14,21.5){\line(1,1){8.2}}  
        \put(-4,31){\line(1,1){8.8}}  
        
          \put(-3.0,-1){\line(0,1){12}}  
          \put(-15,10){\line(0,1){8.2}}. 
          \put(-5.5,20){\line(0,1){8.6}}  
           \put(6,30){\line(0,1){8.4}}  
            \put(17,20){\line(0,1){8.5}}  
            \put(6.8,10){\line(0,1){8.2}}   
            
	\put(5,40){\line(0,1){12}}   
	\put(-7,31){\line(0,1){12}}  
	\put(-16, 21){\line(0,1){12}}. 
	
	 \put(-7,43){\line(1,1){10}}  
	\put(-17,33){\line(1,1){10}}  
	
	\end{picture}
\end{minipage}\\

\begin{tabular}{r|rrrrrrrrrrrrrrr}
$\to$: & $0$ & $1$ & $2$ & $3$ & $4$ & $5$ & $6$ & $7$ & $8$ & $9$ & $10$ & $11$ & $12$ & $13$ & $14$\\
\hline
    $0$ & $1$ & $1$ & $1$ & $1$ & $1$ & $1$ & $1$ & $1$ & $1$ & $1$ & $1$ & $1$ & $1$ & $1$ & $1$ \\
    $1$ & $0$ & $1$ & $2$ & $3$ & $4$ & $5$ & $6$ & $7$ & $8$ & $9$ & $10$ & $11$ & $12$ & $13$ & $14$ \\
    $2$ & $0$ & $1$ & $1$ & $3$ & $4$ & $5$ & $4$ & $5$ & $8$ & $9$ & $8$ & $9$ & $12$ & $13$ & $14$ \\
    $3$ & $4$ & $1$ & $1$ & $1$ & $4$ & $1$ & $4$ & $1$ & $4$ & $1$ & $4$ & $1$ & $1$ & $4$ & $1$ \\
    $4$ & $3$ & $1$ & $2$ & $3$ & $1$ & $1$ & $2$ & $2$ & $9$ & $9$ & $11$ & $11$ & $12$ & $12$ & $12$ \\
    $5$ & $0$ & $1$ & $2$ & $3$ & $4$ & $1$ & $6$ & $2$ & $8$ & $9$ & $10$ & $11$ & $12$ & $13$ & $12$ \\
   $6$ & $3$ & $1$ & $1$ & $3$ & $1$ & $1$ & $1$ & $1$ & $9$ & $9$& $9$ & $9$ & $12$ & $12$ & $12$ \\
    $7$ & $0$ & $1$ & $1$ & $3$ & $4$ & $1$ & $4$ & $1$ & $8$ & $9$ & $8$ & $9$ & $12$ & $13$ & $12$ \\
    $8$ & $12$ & $1$ & $2$ & $12$ & $1$ & $1$ & $2$ & $2$ & $1$ & $1$ & $2$ & $2$ & $12$ & $12$ & $12$ \\
    $9$ & $13$ & $1$ & $2$ & $12$ & $4$ & $1$ & $6$ & $2$ & $4$ & $1$ & $6$ & $2$ & $12$ & $13$ & $12$ \\
    $10$ & $12$ & $1$ & $1$ & $12$ & $1$ & $1$ & $1$ & $1$ & $1$ & $1$ & $1$ & $1$ & $12$ & $12$ & $12$ \\
    $11$ & $13$ & $1$ & $1$ & $12$ & $4$ & $1$ & $4$ & $1$ & $4$ & $1$ & $4$ & $1$ & $12$ & $13$ & $12$ \\
    $12$ & $8$ & $1$ & $1$ & $9$ & $4$ & $5$ & $4$ & $5$ & $8$ & $9$ & $8$ & $9$ & $1$ & $4$ & $5$ \\
    $13$ & $9$ & $1$ & $1$ & $9$ & $1$ & $1$ & $1$ & $1$ & $9$ & $9$ & $9$ & $9$ & $1$ & $1$ & $1$ \\
    $14$ & $8$ & $1$ & $1$ & $9$ & $4$ & $1$ & $4$ & $1$ & $8$ & $9$ & $8$ & $9$ & $1$ & $4$ & $1$
\end{tabular} \\
\ \\ \ \\
\begin{tabular}{r|rrrrrrrrrrrrrrr}
$'$: & $0$ & $1$ & $2$ & $3$ & $4$ & $5$ & $6$ & $7$ & $8$ & $9$ & $10$ & $11$ & $12$ & $13$ & $14$\\
\hline
   & $1$& $0$ & $3$ & $2$ & $8$ & $10$ & $9$ & $11$ & $8$ & $10$ & $9$ & $11$ & $12$ & $1$ & $2$
\end{tabular} \\

\medskip

Here is yet another direction to consider for a possible generalization of Theorem \ref{Level_Theorem}.
Recall that $\mathbf{BDQDSt}$ denotes the subvariety of $\mathbf{DQD}$ defined by \\

\begin{itemize}
\item[(B)] $(x \lor x^*)' \approx x' \land  x{^*}$  
\item[(St)] $x^*  \lor x{^{**}} \approx 1$.\\
\end{itemize}

In this context, we have the following theorem, with M. Kinyon, which will be published in a future paper.\\

\begin{Theorem}
  The variety $\mathbf{BDQDSt}$ is at level $2$. \\ 
\end{Theorem}

\small

\ \\
\ \ \ \ \ \ \ \ \\
Department of Mathematics\\
State University of New York\\
New Paltz, NY 12561\\
\
\\
sankapph@newpaltz.edu\\
\ \

\vspace{1cm}
\begin{thebibliography}{99}



\bibitem{BaDw74}  R. Balbes and PH. Dwinger, {\bf Distributive lattices}, Univ. of Missouri Press, Columbia, 1974.

\bibitem{BuSa81}  S. Burris and H.P. Sankappanavar, {\bf A course in universal algebra}, Springer--Verlag, New York, 1981.  The free, corrected version (2012) is available online as a PDF file at 
 {\sf math.uwaterloo.ca/$\sim$snburris}

\bibitem{CoSa18} J. M. Cornejo and H. P. Sankappanavar, {\it De Morgan semi-Heyting logic}, 
 (2018).  Submitted.


\bibitem{Le70}
Lee, K.B., Equational classes of distributive pseudo-complemented lattices,
Canad. J. Math. 22 (1970), 881--891.

\bibitem{Mo80} A. Monteiro, {\it
Sur les algebres de Heyting symetriques}, Portugaliae Mathemaica 39 (1980), 1--237.


\bibitem{Mc} W. McCune, {\it Prover9 and Mace 4, http://www.cs.unm.edu/mccune/prover9/}

\bibitem{Ra74} H. Rasiowa, {\bf
An algebraic approach to non-classical logics}, North--Holland Publ.Comp., Amsterdam, (1974).


\bibitem{Sa85} H.P. Sankappanavar, {\it Heyting algebras with dual pseudocomplementation},  Pacific J. Math. 117 (1985), 405--415.

\bibitem{Sa87} H.P. Sankappanavar, {\it Heyting algebras with a dual lattice endomorphism},  Zeitschr. f. math. Logik und Grundlagen d. Math. 33 (1987), 565--573.

\bibitem{Sa87a} H.P. Sankappanavar, {\it Semi-De Morgan algebras},  J. Symbolic. Logic 52  (1987), 712--724.

\bibitem{Sa07} H.P. Sankappanavar, {\it Semi--Heyting algebras: An abstraction from
Heyting algebras}, Actas del IX Congreso Dr. A. Monteiro (2007), 33-66.


\bibitem{Sa11} H.P. Sankappanavar, {\it Expansions of Semi-Heyting algebras. I: Discriminator varieties}, Studia Logica 98 (1-2) (2011), 27-81.

\bibitem{Sa14} H.P. Sankappanavar, {\it Dually quasi-De Morgan Stone Semi-Heyting algebras I.  Regularity}, Categories and General Algebraic Structures with Applications, 2 (2014), 55-75.

\bibitem{Sa14a} H.P. Sankappanavar, {\it Dually quasi-De Morgan Stone Semi-Heyting algebras II}, Categories and General Algebraic Structures with Applications, 2 (2014), 77-99.

\bibitem{Sa15} H.P. Sankappanavar, {\it A note on regular De Morgan semi-Heyting algebras}, Demonstrtio Mathematica 49 (2016), 252-265.  (ArXiv.org, November 2014),   


\bibitem{Sa17} H.P. Sankappanavar, {\it JI-distributive dually quasi-De Morgan semi-Heyting and Heyting algebras.} (Submitted).


\bibitem{Sa18} H.P. Sankappanavar, {\it Regular dually Stone semi-Heyting algebras}. In Preparation.

\end{thebibliography}
\end{document}